\documentclass[11pt]{amsart}

\usepackage{amsmath,amssymb,amsfonts,amsthm}
\usepackage{hyperref}
\usepackage{mathrsfs}
\usepackage{xcolor}
\usepackage{comment}

\numberwithin{equation}{section}

\newtheorem{theorem}{Theorem}[section]

\newtheorem{lemma}[theorem]{Lemma}
\newtheorem{corollary}[theorem]{Corollary}
\theoremstyle{remark}
\newtheorem{remark}[theorem]{Remark}

\DeclareMathOperator{\Ric}{Ric}
\DeclareMathOperator{\Sec}{Sec}

\newcommand{\II}{\mathrm{I\!I}}

\newcommand{\abs}[1]{\left|#1\right|}
\newcommand{\R}{\mathbb R}
\newcommand{\Sph}{\mathbb S}

\begin{document}


\title[Stable free boundary CMC surfaces]
{Topology of stable free boundary CMC surfaces under lower Ricci curvature bounds}
\author[R. Antonia]{Railane Antonia}
\author[M. P. Cavalcante]{Marcos P. Cavalcante}
\author[V. Souza]{Vinicius Souza}

\address{Institute of Mathematics, Federal University of Alagoas (UFAL), Macei\'o--AL, Brazil}
\email{railane.silva@academico.ufpb.br}
\email{marcos@pos.mat.ufal.br}
\email{vinicius.souza@im.ufal.br}

\date{\today}

\begin{abstract}
We establish intrinsic area--length--topology inequalities for compact free boundary constant mean curvature (CMC) surfaces in three-manifolds with Ricci curvature bounded from below.

Our main result is obtained from a conformal upper bound for a constrained first Robin eigenvalue of the Jacobi operator, derived via a balancing argument. This yields a quantitative inequality that does not require stability and captures both interior and boundary contributions.

As an application, we obtain explicit topological restrictions for stable free boundary CMC surfaces under a natural curvature pinching condition. In particular, in weakly convex domains, stability forces low topological complexity, with genus at most three and a small number of boundary components.

These results show that effective topological control persists even in negatively curved settings, where classical rigidity phenomena are no longer available.
\end{abstract}

\maketitle

\section{Introduction}

Constant mean curvature (CMC) surfaces arise naturally as critical points of the area functional under a volume constraint. More precisely, a closed surface $\Sigma^2\subset (M^3,g)$ is said to have constant mean curvature if it is a critical point of the area functional among variations preserving the enclosed volume. 

The second variation leads to the notion of stability: a CMC surface is said to be \emph{stable} if the second variation of area is nonnegative for all volume--preserving normal variations. The associated Morse index counts the number of independent directions along which the area decreases.

A fundamental result in this direction is the classical theorem of Barbosa--do Carmo--Eschenburg~\cite{BdcE}, which asserts that the only closed stable CMC surfaces in space forms are the totally umbilic spheres. This illustrates a general principle: in highly symmetric ambient spaces, stability imposes strong rigidity and often leads to classification.

Beyond space forms, several works have investigated quantitative and topological consequences of stability for closed CMC surfaces. 
A fundamental contribution in this direction is due to El Soufi--Ilias~\cite{ElSoufiIlias1992}, who, building on the conformal invariant introduced by Li--Yau~\cite{LiYau1982}, obtained conformal eigenvalue estimates with applications to the stability and area of closed
CMC surfaces in manifolds with nonnegative Ricci curvature.

Frensel~\cite{Frensel} proved that if $\Sigma^2$ is a compact orientable surface of genus $\gamma$ immersed with constant mean curvature $H$ in a $3$--manifold $M^3$, and if the immersion is stable and satisfies
\begin{equation}\label{eq:h-pinching}
   H^2> -2\inf_\Sigma \Ric_M,
\end{equation}
then $\gamma\le 3$. Here $H=k_1+k_2$ denotes the (non-normalized) mean curvature used in the present paper. In addition, she obtained sharp area bounds and rigidity in the equality case.

We also refer to \cite{NelliRosenberg2006,Souam2010,Souam2015} for further developments, including refined genus bounds and classification results in product spaces such as $\mathbb H^2\times\mathbb R$.

\medskip

In this paper we focus primarily on the \emph{free boundary} setting. Let $(M^3,g)$ be a Riemannian $3$--manifold and let $\Omega\subset M$ be a bounded domain with smooth boundary. A compact surface $\Sigma^2\subset \Omega$ is a \emph{free boundary CMC surface} if its mean curvature is constant, say $H_\Sigma\equiv H$, and
\[
\partial\Sigma \subset \partial\Omega,\quad \text{and}\quad
\Sigma \text{ meets } \partial\Omega \text{ orthogonally along }\partial\Sigma.
\]
Such surfaces arise as critical points of the area functional among volume-preserving variations with free boundary on $\partial\Omega$; see Ros--Vergasta~\cite{RosVergasta1995}, Nunes~\cite{Nunes2017}, and the book of L\'opez~\cite{Lopez-book} for background.

A free boundary CMC surface is said to be \emph{stable} if the second variation of area is nonnegative for all volume--preserving normal variations with boundary free to move on $\partial\Omega$.

\medskip
A central theme is to understand how ambient curvature and boundary convexity impose both qualitative and quantitative restrictions on stable solutions. Heuristically, stability tends to select the simplest critical points, and this often manifests as rigidity in symmetric environments.

A classical prototype is Nitsche's theorem~\cite{Nitsche}, asserting that free boundary CMC disks in the Euclidean ball are spherical caps. 
In geodesic balls of space forms, foundational rigidity and classification results due to Ros--Vergasta~\cite{RosVergasta1995} and Souam~\cite{Souam1997} show that stable free boundary CMC surfaces are highly restricted. 
In dimension two, Nunes~\cite{Nunes2017} completed the classification in the Euclidean ball by ruling out the remaining genus one case in the theorem of Ros--Vergasta, thereby showing that stable free boundary CMC surfaces are spherical caps.

In this direction, the work of Wang--Xia~\cite{WangXia} plays a particularly important role. They established rigidity results for stable free boundary and capillary CMC hypersurfaces in balls of space forms, and their arguments apply in all dimensions, including dimension two. These results further reinforce the principle that, in model spaces with strong symmetry and convex boundary, stability forces rigidity and classification.

\medskip

In the \emph{free boundary minimal} case ($H=0$), another manifestation
of this philosophy is the emergence of area--length--topology inequalities
obtained from the stability inequality combined with the Gauss equation and
the Gauss--Bonnet theorem. This circle of ideas goes back, in the closed
case, to the work of Schoen--Yau on stable minimal surfaces and scalar
curvature~\cite{SchoenYau}. In the free boundary setting,
Ambrozio~\cite{Ambrozio} proved that, under lower bounds on the scalar
curvature \(R_M\) and on the mean curvature \(H_{\partial\Omega}\) of the
boundary, a two-sided stable free boundary minimal surface satisfies
\begin{equation}\label{eq:ambrozioinequality}
    \frac{1}{2}\inf_{\Omega} R_M\,|\Sigma|
    + \inf_{\partial\Omega} H_{\partial\Omega}\,|\partial\Sigma|
    \le 2\pi\chi(\Sigma).
\end{equation}
This inequality has been further refined and generalized in subsequent
works, including Mendes~\cite{Mendes2018}. In the capillary setting,
Longa~\cite{Longa2022} introduced the corresponding angle-dependent
functional and proved analogous stability estimates, as well as low-index
bounds involving the genus, the number of boundary components, and the area.
These estimates are sharp in several geometric settings and lead to rigidity
in the equality case under additional curvature assumptions.

These ideas were further developed in the low-index theory of
Ambrozio--Buzano--Carlotto--Sharp~\cite{ABCS}, where a detailed analysis of
free boundary minimal surfaces with bounded index was carried out, leading to
strong structural and compactness results under mild curvature assumptions.
In the same nonnegative curvature regime, Carlotto--Franz~\cite{CarlottoFranz}
investigated the interplay between area, topology, and index, establishing
quantitative inequalities and compactness results, and obtaining in particular
disk-type rigidity.

A different but closely related line of research concerns lower bounds for the
Morse index. In the closed case, the index of a compact CMC surface in
\(\mathbb R^3\) or \(\mathbb S^3\) was shown in~\cite{CO1} to grow at least
linearly with the genus. In the free boundary setting, analogous lower bounds
depending on both the genus and the number of boundary components were obtained
in~\cite{CO2}. These results were later extended to more general ambient
manifolds by Aiex--Hong~\cite{AiexHong}, and further developments in the
capillary setting were obtained by Hong--Saturnino~\cite{HongSaturnino}, who
established structural results such as finite topology and conformal
compactification under finite index. These contributions play an important role
in compactness and bubbling theories.

\smallskip

The goal of this paper is to establish area--length--topology inequalities for
free boundary CMC surfaces under lower Ricci curvature bounds that include the
hyperbolic normalization \(\Ric_M\ge -2\). In contrast with much of the existing
theory, which relies on nonnegative curvature assumptions to obtain rigidity,
we allow negative Ricci curvature and impose no symmetry assumptions.

Instead of relying directly on the stability inequality with constant test
functions, we estimate the constrained first Robin eigenvalue of the Jacobi
operator by means of conformal test functions. Our main tool is an intrinsic
conformal estimate based on the balancing method of
Ros--Vergasta~\cite{RosVergasta1995}, in the version developed by
Chen--Fraser--Pang~\cite{ChenFraserPang}. This approach produces quantitative
area--length--topology inequalities without assuming stability, while retaining
explicit interior and boundary contributions. When combined with stability,
weak boundary convexity, and the natural pinching condition \(H^2\ge 4\), it
yields effective topological restrictions for free boundary CMC surfaces.

\medskip
Throughout the paper, we work under the assumption that
\[
\Ric_M \ge -2
\qquad\text{on } \Omega.
\]
We measure the possible lack of convexity of $\partial\Omega$ by setting
\begin{equation}\label{eq:def-tau-boundary}
\tau_{\partial\Omega}
:=
\inf\Big\{
\II^{\partial\Omega}(v,v)
:\ p\in\partial\Omega,\ v\in T_p(\partial\Omega),\ |v|=1
\Big\}.
\end{equation}
Thus $\tau_{\partial\Omega}\ge 0$ precisely when $\partial\Omega$ is weakly
convex.

\medskip

\noindent\textbf{Main quantitative estimate.}
Let $J=\Delta+V$ be the Jacobi operator of a two-sided free boundary CMC surface,
with
\[
V=\Ric_M(N,N)+|A|^2,
\qquad
W=\II^{\partial\Omega}(N,N),
\]
and let $\bar\mu_1$ be the constrained first Robin eigenvalue of the pair $(J,B)$ associated with these potentials (see \eqref{eq:mu1bar-def} below). Our first theorem is an
area--length--topology inequality which does \emph{not} assume stability.

\begin{theorem}\label{thm:arealength}
Let $(M^3,g)$ be a Riemannian $3$-manifold and let
$\Omega\subset M$ be a bounded domain with smooth boundary.
Assume that
\[
\Ric_M \ge -2
\qquad\text{on } \Omega,
\]
and let $\tau_{\partial\Omega}$ be defined by \eqref{eq:def-tau-boundary}.

Let $\Sigma^2\subset\Omega$ be a compact, two-sided free boundary CMC surface
with constant mean curvature $H$, genus $\gamma$, and $r$ boundary components.
Then
\begin{equation}\label{eq:areatop-intro-new}
\bigl(\bar\mu_1+H^2-4\bigr)|\Sigma|
+
3\tau_{\partial\Omega}|\partial\Sigma|
<
8\pi\Big\lfloor\frac{\gamma+3}{2}\Big\rfloor
+
4\pi\chi(\Sigma),
\end{equation}
where $\chi(\Sigma)=2-2\gamma-r$ is the Euler characteristic of $\Sigma$.
\end{theorem}
\medskip

\noindent\textbf{Topological restrictions for stable solutions in convex domains.}
When $\Sigma$ is stable for volume-preserving variations one has
$\bar\mu_1\ge 0$. If, in addition, $\partial\Omega$ is weakly convex
(equivalently, $\tau_{\partial\Omega}\ge 0$) and the pinching condition
\[
H^2\ge 4
\]
holds, then the left-hand side of \eqref{eq:areatop-intro-new} is nonnegative
and the inequality becomes purely topological, forcing strong restrictions on
the pair $(\gamma,r)$.

\begin{theorem}\label{thm:convex}
Let $(M^3,g)$ be a Riemannian $3$-manifold and let
$\Omega\subset M$ be a bounded domain with smooth boundary.
Assume that
\[
\Ric_M \ge -2
\qquad\text{on } \Omega,
\]
and that $\partial\Omega$ is weakly convex, i.e.,
$\II^{\partial\Omega}\ge 0$ with respect to the outward unit normal.

Let $\Sigma^2\subset\Omega$ be a compact, two-sided, stable free boundary CMC
surface with constant mean curvature $H$. If
\[
H^2\ge 4,
\]
then the only possibilities for the genus $\gamma$ of $\Sigma$ and the number
$r$ of boundary components of $\partial\Sigma$ are
\[
\gamma\in\{0,1\}\ \text{ with }\ r\le 3,
\qquad\text{or}\qquad
\gamma\in\{2,3\}\ \text{ with }\ r=1.
\]
In particular, $\gamma\le 3$ and $r\le 3$.
\end{theorem}

\begin{remark}
If \(\inf_\Omega \Ric_M>-2\), then \(H^2\ge 4\) implies Frensel's strict
pinching condition. The borderline case \(\inf_\Omega\Ric_M=-2\) and
\(H^2=4\) is also allowed in our theorem, reflecting the additional boundary
terms present in the free boundary setting.
\end{remark}

\begin{corollary}\label{cor:area-estimates-by-topology}
Under the assumptions of Theorem~\ref{thm:convex}, one has
\[
(H^2-4)|\Sigma| < 12\pi,
\]
and more precisely
\[
(H^2-4)|\Sigma|
<
\begin{cases}
12\pi, & r=1\ \text{and}\ \gamma\in\{0,1\},\\
8\pi, & r=2\ \text{and}\ \gamma\in\{0,1\},\\
4\pi, & \text{otherwise}.
\end{cases}
\]
\end{corollary}
\medskip

Theorems~\ref{thm:arealength} and~\ref{thm:convex} provide a quantitative
counterpart, in a Ricci curvature regime that includes negative curvature,
to the low topological complexity phenomena usually obtained under
nonnegative curvature assumptions. Although stability alone is not expected
to imply classification in this setting, it still yields effective topological
control when combined with the natural pinching condition
\[
H^2\ge 4.
\]
In particular, Theorem~\ref{thm:convex} gives explicit restrictions on the
genus and the number of boundary components and applies to geodesic balls in
hyperbolic space $\mathbb H^3$, where $\Ric=-2$.

\medskip

\noindent\emph{Organization.}
Section~\ref{sec:stability-spectrum} recalls the Jacobi operator for free boundary CMC surfaces and introduces the constrained Robin eigenvalue.
Section~\ref{sec:conformal-bound} proves the intrinsic conformal bound.
Theorems~\ref{thm:arealength} and~\ref{thm:convex} are proved in Section~\ref{sec:proofs}, where the necessary geometric estimates are also developed.
Section~\ref{sec:scalar-curvature-version} presents a scalar curvature variant.
We conclude with examples (Section~\ref{sec:examples}).

\section{Stability, the Jacobi operator, and a constrained Robin spectrum}
\label{sec:stability-spectrum}

Let $\Sigma^2\subset \Omega\subset M^3$ be a compact, two--sided free boundary CMC surface.
Fix a global unit normal $N$ along $\Sigma$ and adopt the convention
\[
H=\mathrm{tr}(A)=k_1+k_2,
\]
where $A$ is the shape operator and $k_1,k_2$ are the principal curvatures.

Let $\II^{\partial\Omega}$ denote the second fundamental form of $\partial\Omega$ with respect to the
outward unit normal. The second variation of area under volume--preserving normal variations
$X=uN$ (see, e.g., \cite{R06,Nunes2017,WangXia}) is given by the index form
\begin{equation}\label{eq:index-form}
\mathcal I(u,u)=
\int_\Sigma |\nabla u|^2
-\int_\Sigma \big(\Ric_M(N,N)+|A|^2\big)u^2
-\int_{\partial\Sigma}\II^{\partial\Omega}(N,N)\,u^2,
\end{equation}
defined for $u\in H^1(\Sigma)$ subject to the volume constraint
\begin{equation}\label{eq:mean-zero}
\int_\Sigma u=0.
\end{equation}

Set
\[
V:=\Ric_M(N,N)+|A|^2,
\qquad
W:=\II^{\partial\Omega}(N,N),
\]
and consider the Jacobi operator $J=\Delta+V$ together with the Robin boundary operator
$Bu:=\partial_\nu u-Wu$, where $\nu$ is the outward unit conormal along $\partial\Sigma$
tangent to $\Sigma$. Define the associated quadratic form
\begin{equation}\label{eq:QJ}
Q[u]
:=
\int_\Sigma |\nabla u|^2-\int_\Sigma V u^2-\int_{\partial\Sigma}W u^2,
\qquad u\in H^1(\Sigma).
\end{equation}
Then $\mathcal I(u,u)=Q[u]$ for every admissible $u$ satisfying \eqref{eq:mean-zero}.

\medskip

Since admissible variations satisfy \eqref{eq:mean-zero}, we introduce the \emph{constrained} first
eigenvalue of $(J,B)$ by
\begin{equation}\label{eq:mu1bar-def}
\bar\mu_1
:=
\inf\left\{
\frac{Q[u]}{\int_\Sigma u^2}\ :\
u\in H^1(\Sigma)\setminus\{0\},\ \int_\Sigma u=0
\right\}.
\end{equation}
Equivalently, $\bar\mu_1$ is the bottom of the spectrum of $Q$ restricted to the codimension--one subspace
\[
H^1_\ast(\Sigma):=\Big\{u\in H^1(\Sigma):\int_\Sigma u=0\Big\}.
\]

\begin{lemma}\label{lem:stable-iff-mu1bar}
The surface $\Sigma$ is stable for volume--preserving variations if and only if
$\bar\mu_1\ge 0$.
\end{lemma}

\begin{proof}
By definition, stability means $Q[u]\ge 0$ for all $u\in H^1_\ast(\Sigma)$.
If $\Sigma$ is stable, then $\frac{Q[u]}{\int_\Sigma u^2}\ge 0$ for every nonzero
$u\in H^1_\ast(\Sigma)$, hence $\bar\mu_1\ge 0$.

Conversely, if $\bar\mu_1\ge 0$, then for any $u\in H^1_\ast(\Sigma)\setminus\{0\}$,
\[
\frac{Q[u]}{\int_\Sigma u^2}\ge \bar\mu_1\ge 0,
\]
so $Q[u]\ge 0$. Thus $\Sigma$ is stable.
\end{proof}

\section{A conformal upper bound for the constrained Robin eigenvalue}
\label{sec:conformal-bound}

In this section, we adapt the approach developed in our previous work~\cite{AntoniaCavalcanteSouza} to derive a geometric upper bound for $\bar\mu_1$.

Let $\widehat\Sigma$ be the closed surface obtained by conformally capping each boundary component
of $\Sigma$ by a disc. Then $\widehat\Sigma$ has the same genus $\gamma$ as $\Sigma$.
A Hersch-type balancing argument on $\widehat\Sigma$ yields the following lemma of
Chen, Fraser and Pang~\cite[Lemma 5.1]{ChenFraserPang}, applied here with $h\equiv 1$.

\begin{lemma}[Chen--Fraser--Pang]\label{lem:hersch-constant}
There exists a conformal map $\Phi:\Sigma\to \Sph^2\subset\R^3$ such that, writing
$\Phi=(\phi_1,\phi_2,\phi_3)$,
\[
\int_\Sigma \phi_i=0,\qquad i=1,2,3,
\qquad\text{and}\qquad
\deg(\Phi)\le \Big\lfloor\frac{\gamma+3}{2}\Big\rfloor.
\]
\end{lemma}

As in~\cite{AntoniaCavalcanteSouza,ChenFraserPang}, the conformality of $\Phi$ implies that the
Dirichlet energy satisfies
\begin{equation}\label{eq:Dir-energy}
\sum_{i=1}^3\int_{\Sigma}|\nabla\phi_i|^2
<
\sum_{i=1}^3\int_{\widehat\Sigma}|\nabla\phi_i|^2
\le
8\pi\Big\lfloor\frac{\gamma+3}{2}\Big\rfloor,
\end{equation}
since the energy of a conformal map to $\Sph^2$ equals $8\pi$ times the degree, and restriction from
$\widehat\Sigma$ to $\Sigma$ strictly decreases energy (as $\partial\Sigma\neq\emptyset$).

\begin{theorem}\label{thm:mu1bar-upper}
Let $(\Sigma^2,ds^2)$ be a compact oriented surface with nonempty boundary and genus $\gamma$.
Let $V\in L^\infty(\Sigma)$ and $W\in L^\infty(\partial\Sigma)$, and let $Q$ be given by \eqref{eq:QJ}.
Then the constrained eigenvalue $\bar\mu_1$ defined by \eqref{eq:mu1bar-def} satisfies
\begin{equation}\label{eq:mu1bar-bound}
\bar\mu_1\,\abs{\Sigma}
\;<\;
8\pi\Big\lfloor\frac{\gamma+3}{2}\Big\rfloor
\;-\;\int_\Sigma V\;-\;\int_{\partial\Sigma}W.
\end{equation}
\end{theorem}

\begin{proof}
By Lemma~\ref{lem:hersch-constant}, each $\phi_i$ satisfies $\int_\Sigma \phi_i=0$, hence
$\phi_i\in H^1_\ast(\Sigma)$ and is admissible in \eqref{eq:mu1bar-def}. Therefore,
\[
\bar\mu_1\int_\Sigma \phi_i^2
\le
\int_\Sigma|\nabla\phi_i|^2-\int_\Sigma V\phi_i^2-\int_{\partial\Sigma}W\phi_i^2.
\]
Summing over $i=1,2,3$ and using $\sum_i\phi_i^2\equiv 1$ on $\Sigma$ and on $\partial\Sigma$ yields
\[
\bar\mu_1\,\abs{\Sigma}
\le
\sum_{i=1}^3\int_\Sigma|\nabla\phi_i|^2
-\int_\Sigma V
-\int_{\partial\Sigma}W.
\]
Finally, apply \eqref{eq:Dir-energy} to obtain the strict inequality \eqref{eq:mu1bar-bound}.
\end{proof}

\section{Proofs of the main theorems}
\label{sec:proofs}

\subsection{Proof of Theorem~\ref{thm:arealength}}

Let $\Sigma^2\subset\Omega$ be as in Theorem~\ref{thm:arealength}. Applying
Theorem~\ref{thm:mu1bar-upper} to the Jacobi data $(V,W)$, where
\[
V=\Ric_M(N,N)+|A|^2,
\qquad
W=\II^{\partial\Omega}(N,N),
\]
gives
\begin{equation}\label{eq:arealength-step1}
\bar\mu_1|\Sigma|
<
8\pi\Big\lfloor \frac{\gamma+3}{2}\Big\rfloor
-\int_\Sigma V
-\int_{\partial\Sigma} W .
\end{equation}

Let $\{e_1,e_2,N\}$ be a local orthonormal frame along $\Sigma$, with
$\{e_1,e_2\}$ tangent to $\Sigma$. The Gauss equation gives
\[
K=\Sec_M(e_1,e_2)+\frac{H^2-|A|^2}{2},
\]
so that
\[
|A|^2=H^2+2\Sec_M(e_1,e_2)-2K.
\]
Since
\[
\Ric_M(e_1,e_1)+\Ric_M(e_2,e_2)
=
2\Sec_M(e_1,e_2)+\Ric_M(N,N),
\]
we obtain
\[
V
=
\Ric_M(N,N)+|A|^2
=
H^2-2K+\Ric_M(e_1,e_1)+\Ric_M(e_2,e_2).
\]
Using the assumption $\Ric_M \ge -2$, we have
\[
\Ric_M(e_1,e_1)+\Ric_M(e_2,e_2)\ge -4,
\]
and therefore
\begin{equation}\label{eq:V-lower}
V\ge H^2-4-2K.
\end{equation}

Using \eqref{eq:V-lower} and the bound
$W=\II^{\partial\Omega}(N,N)\ge\tau_{\partial\Omega}$ along $\partial\Sigma$
in \eqref{eq:arealength-step1}, we obtain
\begin{equation}\label{eq:arealength-step2}
(\bar\mu_1+H^2-4)|\Sigma|
+\tau_{\partial\Omega}|\partial\Sigma|
<
8\pi\Big\lfloor \frac{\gamma+3}{2}\Big\rfloor
+2\int_\Sigma K .
\end{equation}

By the Gauss--Bonnet theorem,
\[
\int_\Sigma K
=
2\pi\chi(\Sigma)-\int_{\partial\Sigma}\kappa_g .
\]
Along $\partial\Sigma$, the free boundary condition implies that
\[
\kappa_g=\II^{\partial\Omega}(T,T),
\]
where $T$ is the unit tangent vector field to $\partial\Sigma$. In particular,
\[
\kappa_g\ge \tau_{\partial\Omega}
\qquad\text{on }\partial\Sigma.
\]
Hence
\[
2\int_\Sigma K
\le
4\pi\chi(\Sigma)-2\tau_{\partial\Omega}|\partial\Sigma|.
\]
Combining this with \eqref{eq:arealength-step2} gives
\[
(\bar\mu_1+H^2-4)|\Sigma|
+
3\tau_{\partial\Omega}|\partial\Sigma|
<
8\pi\Big\lfloor \frac{\gamma+3}{2}\Big\rfloor
+
4\pi\chi(\Sigma),
\]
which is \eqref{eq:areatop-intro-new}.
\qed

\subsection{Proof of Theorem~\ref{thm:convex}}

If $\Sigma$ is stable, Lemma~\ref{lem:stable-iff-mu1bar} gives
$\bar\mu_1\ge 0$. Since $\partial\Omega$ is weakly convex, we have
$\tau_{\partial\Omega}\ge 0$. Under the assumption
\[
H^2 \ge 4,
\]
the left-hand side of \eqref{eq:areatop-intro-new} is nonnegative. Hence
\[
0
<
8\pi\Big\lfloor\frac{\gamma+3}{2}\Big\rfloor
+
4\pi\chi(\Sigma).
\]
Equivalently,
\[
0< 
2\Big\lfloor\frac{\gamma+3}{2}\Big\rfloor+\chi(\Sigma).
\]
Since $\chi(\Sigma)=2-2\gamma-r$, this inequality is equivalent to
\[
r
<
2\Big\lfloor\frac{\gamma+3}{2}\Big\rfloor+2-2\gamma.
\]
Checking separately the cases $\gamma$ even and odd gives precisely
\[
\gamma\in\{0,1\}\ \text{ with }\ r\le 3,
\qquad\text{or}\qquad
\gamma\in\{2,3\}\ \text{ with }\ r=1.
\]
In particular, $\gamma\le 3$ and $r\le 3$.

\qed

\section{A scalar curvature estimate}
\label{sec:scalar-curvature-version}

We record a variant of the preceding estimate involving the ambient scalar
curvature. This follows from the standard Schoen--Yau \cite{SchoenYau} rearrangement of the
Gauss equation.
\begin{equation}\label{eq:ric-identity}
-|A|^2
=
2\Ric_M(N,N)-R_M-H^2+2K.
\end{equation}

\begin{theorem}\label{thm:areatop-ineq-scalar}
Let $(M^3,g)$ be a Riemannian $3$--manifold and let
$\Omega\subset M$ be a bounded domain with smooth boundary. Assume that
\[
\Ric_M\le 0
\qquad\text{on } \Omega,
\]
and that $\partial\Omega$ is weakly mean-convex, i.e.,
$H^{\partial\Omega}\ge 0$ with respect to the outward unit normal.

Let $\Sigma^2\subset\Omega$ be a compact, two-sided free boundary CMC surface
with constant mean curvature $H$, genus $\gamma$, and $r$ boundary components.
Then
\begin{equation}\label{eq:areatop-ineq-scalar}
(\bar\mu_1+H^2+\inf_\Omega R_M)|\Sigma|
+\tau_{\partial\Omega}|\partial\Sigma|
<
8\pi\Big\lfloor\frac{\gamma+3}{2}\Big\rfloor
+4\pi\chi(\Sigma).
\end{equation}
\end{theorem}

\begin{proof}
Applying Theorem~\ref{thm:mu1bar-upper} to the Jacobi data
\[
V=\Ric_M(N,N)+|A|^2,
\qquad
W=\II^{\partial\Omega}(N,N),
\]
we obtain
\begin{equation}\label{eq:scalar-step1}
\bar\mu_1|\Sigma|
<
8\pi\Big\lfloor\frac{\gamma+3}{2}\Big\rfloor
-\int_\Sigma V
-\int_{\partial\Sigma} W .
\end{equation}
Using \eqref{eq:ric-identity} and the assumption $\Ric_M\le 0$, we get
\[
-\int_\Sigma V
\le
-\int_\Sigma R_M
-H^2|\Sigma|
+2\int_\Sigma K .
\]
Substituting into \eqref{eq:scalar-step1} yields
\[
\bar\mu_1|\Sigma|
<
8\pi\Big\lfloor\frac{\gamma+3}{2}\Big\rfloor
-\int_\Sigma R_M
-H^2|\Sigma|
+2\int_\Sigma K
-\int_{\partial\Sigma} \II^{\partial\Omega}(N,N) .
\]
By Gauss--Bonnet,
\[
2\int_\Sigma K
=
4\pi\chi(\Sigma)-2\int_{\partial\Sigma}\kappa_g.
\]
Along $\partial\Sigma$, the free boundary condition gives
\[
\kappa_g=\II^{\partial\Omega}(T,T),
\qquad
\kappa_g+\II^{\partial\Omega}(N,N)=H^{\partial\Omega}.
\]
Hence
\[
-2\kappa_g - \II^{\partial\Omega}(N,N)
= -\kappa_g - H^{\partial\Omega}.
\]
Using $H^{\partial\Omega}\ge 0$ and
$\kappa_g\ge\tau_{\partial\Omega}$, we obtain
\[
\bar\mu_1|\Sigma|
<
8\pi\Big\lfloor\frac{\gamma+3}{2}\Big\rfloor
-(\inf_\Omega R_M)|\Sigma|
-H^2|\Sigma|
+4\pi\chi(\Sigma)
-\tau_{\partial\Omega}|\partial\Sigma|.
\]
Rearranging gives \eqref{eq:areatop-ineq-scalar}.
\end{proof}

\begin{corollary}
Under the assumptions of Theorem~\ref{thm:areatop-ineq-scalar}, suppose in
addition that
\[
\bar\mu_1+H^2+\inf_\Omega R_M\ge 0
\qquad\text{and}\qquad
\tau_{\partial\Omega}\ge 0.
\]
Then
\[
\gamma\in\{0,1\}\ \text{ with } r\le 3,
\qquad\text{or}\qquad
\gamma\in\{2,3\}\ \text{ with } r=1.
\]
In particular, $\gamma\le 3$ and $r\le 3$.
\end{corollary}

\section{Examples of stable free boundary CMC surfaces in nonpositive curvature}
\label{sec:examples}

In ambient spaces of nonpositive curvature, explicit free boundary CMC examples are available, yet stability often forces strong rigidity and, in many settings, significant symmetry.
We briefly recall some families that serve as test cases and help calibrate the scope of our results.

\medskip
\noindent\textbf{Hyperbolic balls.}
In geodesic balls of the hyperbolic space $\mathbb H^3$, the most basic \emph{stable} free boundary CMC examples are the totally umbilical ones, namely spherical caps (including the totally geodesic disk).
Such caps are singled out by rigidity and classification results in space forms; see  Wang--Xia~\cite{WangXia} and the references therein.

\smallskip
Beyond the umbilical examples, there are explicit \emph{non-umbilical} free boundary CMC surfaces in hyperbolic balls.
Cerezo--Fern\'andez--Mira~\cite{AlbertoIsabelPablo} construct infinitely many compact free boundary CMC annuli in geodesic balls of $\mathbb H^3$,
including embedded non-rotational examples for suitable values of $H$.
More recently, Cerezo~\cite{Cerezo2025} constructed a countable collection of one--parameter families of \emph{non-rotational} free boundary minimal annuli ($H=0$),
each family sharing a prismatic symmetry group and arising as a bifurcation from certain free boundary hyperbolic catenoids.
These constructions show that, even in the highly symmetric setting of a ball in $\mathbb H^3$, the landscape of compact free boundary CMC surfaces is not exhausted by spherical caps once one drops stability.

\medskip
\noindent\textbf{Product geometries and slabs.}
A particularly flexible source of examples in nonpositive curvature arises in product manifolds and slabs.
In $\mathbb H^n\times[0,l]$ (and more generally in $M^n\times[0,l]$), vertical cylinders over closed CMC hypersurfaces in $M^n$ yield free boundary CMC hypersurfaces meeting the slices orthogonally.
Souam~\cite{Souam2021} proves that such ``tubes'' satisfy an explicit stability criterion.
In particular, a tube of radius $\rho$ and height $l$ in $\mathbb H^n\times[0,l]$ is stable if and only if
\[
\pi\sinh\rho \ \ge\ \sqrt{n-1}\,l,
\]
so in $\mathbb H^2\times[0,l]$ one obtains stable free boundary CMC cylinders whenever $\pi\sinh\rho\ge l$.

\smallskip
Moreover, Souam~\cite{Souam2021} establishes a structural dichotomy for stable free boundary CMC immersions
$\psi:\Sigma\to M\times[0,l]$ connecting the two boundary components $M\times\{0\}$ and $M\times\{l\}$:
either $\psi(\Sigma)$ is a vertical cylinder, or it is a vertical graph (locally, and under additional hypotheses globally).
In the graphical case, Souam shows that the lowest Dirichlet eigenvalue of the stability operator equals $0$.
These results make precise the principle that stability in nonpositive curvature often forces global graphical behavior or a product structure, and they also yield nonexistence statements for sufficiently wide slabs.

\medskip
\noindent\textbf{Capillary examples supported by umbilical barriers.}
In $\mathbb H^3$, one also finds stable capillary (and in particular free boundary) examples supported by totally umbilical barriers.
In the capillary framework (with free boundary as the right--angle case), pieces of horospheres provide natural strongly stable examples in suitable domains bounded by totally geodesic planes or equidistant surfaces.
We refer to Bueno--L\'opez~\cite{BuenoLopez} for stability results and an instability/bifurcation analysis for Killing cylinders, and to~\cite{BuenoLopez2} for a related extension.
Together, these examples illustrate that while nonpositive curvature admits nontrivial families of free boundary CMC surfaces, stability tends to impose sharp restrictions, often pushing solutions toward umbilical or cylindrical configurations.

\section*{Acknowledgments}
We thank Eduardo Longa, Giada Franz, Han Hong, Isabel Fernandez, and Rabah Souam for their helpful comments and suggestions.
This work was carried out within the CAPES--Math AmSud project \emph{New Trends in Geometric Analysis}
(CAPES, Grant 88887.985521/2024-00).
R.\,A.\ was partially supported by the Brazilian National Council for Scientific and Technological
Development (CNPq, Grant 151804/2024-9).
M.\,C.\ was partially supported by CNPq (Grants 311136/2023-0).
V.\,S.\ was partially supported by FAPEAL (Grant E:60030.0000000971/2024).

\bibliographystyle{amsplain}
\bibliography{bibliography}

\end{document}